\documentclass[11pt]{article}
\usepackage[margin=1in]{geometry} 
\usepackage[utf8x]{inputenc}
\usepackage[OT4]{fontenc}
\usepackage[T1]{fontenc}
\usepackage{amssymb,amsfonts,amsmath,bbm,mathrsfs,stmaryrd}
\usepackage{xcolor}
\usepackage{url}
\usepackage{relsize}
\usepackage{enumerate}
\usepackage{tikz-cd}

\usepackage[colorlinks,
             linkcolor=black!75!red,
             citecolor=blue,
             pdftitle={},
             pdfauthor={x},
             pdfproducer={pdfLaTeX},
             pdfpagemode=FullScreen,
             bookmarksopen=true,
             bookmarksnumbered=true]{hyperref}

\usepackage{tikz}
\usetikzlibrary{arrows,calc,decorations.pathreplacing,decorations.markings,intersections,shapes.geometric,through,fit,shapes.symbols,positioning,decorations.pathmorphing}

\usepackage[amsmath,thmmarks]{ntheorem}
\usepackage{cleveref}
\usepackage{comment}

\creflabelformat{enumi}{label*=\arabic*}

\crefname{section}{Section}{Sections}
\crefformat{section}{#2Section~#1#3} 
\Crefformat{section}{#2Section~#1#3} 

\crefname{subsection}{\S}{\S\S}
\crefformat{subsection}{#2\S#1#3} 
\Crefformat{subsection}{#2\S#1#3} 

\theoremstyle{plain}

\newtheorem{lemma}{Lemma}[section]
\newtheorem{proposition}[lemma]{Proposition}

\newtheorem{theorem}[lemma]{Theorem}

\theoremstyle{nonumberplain}

\theoremstyle{plain}
\theorembodyfont{\upshape}
\theoremsymbol{\ensuremath{\blacklozenge}}

\newtheorem{definition}[lemma]{Definition}

\newtheorem{remark}[lemma]{Remark}

\crefname{definition}{definition}{definitions}
\crefformat{definition}{#2definition~#1#3} 
\Crefformat{definition}{#2Definition~#1#3} 

\crefname{ex}{example}{examples}
\crefformat{example}{#2example~#1#3} 
\Crefformat{example}{#2Example~#1#3} 

\crefname{remark}{remark}{remarks}
\crefformat{remark}{#2remark~#1#3} 
\Crefformat{remark}{#2Remark~#1#3} 

\crefname{convention}{convention}{conventions}
\crefformat{convention}{#2convention~#1#3} 
\Crefformat{convention}{#2Convention~#1#3} 

\crefname{claim}{claim}{claims}
\crefformat{claim}{#2claim~#1#3} 
\Crefformat{claim}{#2Claim~#1#3} 

\crefname{conjecture}{conjecture}{conjectures}
\crefformat{conjecture}{#2conjecture~#1#3} 
\Crefformat{conjecture}{#2Conjecture~#1#3}

\crefname{lemma}{lemma}{lemmas}
\crefformat{lemma}{#2lemma~#1#3} 
\Crefformat{lemma}{#2Lemma~#1#3} 

\crefname{proposition}{proposition}{propositions}
\crefformat{proposition}{#2proposition~#1#3} 
\Crefformat{proposition}{#2Proposition~#1#3} 

\crefname{question}{question}{questions}
\crefformat{question}{#2question~#1#3} 
\Crefformat{question}{#2Question~#1#3}

\crefname{corollary}{corollary}{corollaries}
\crefformat{corollary}{#2corollary~#1#3} 
\Crefformat{corollary}{#2Corollary~#1#3} 

\crefname{theorem}{theorem}{theorems}
\crefformat{theorem}{#2theorem~#1#3} 
\Crefformat{theorem}{#2Theorem~#1#3} 

\crefname{assumption}{assumption}{Assumptions}
\crefformat{assumption}{#2assumption~#1#3} 
\Crefformat{assumption}{#2Assumption~#1#3} 

\crefname{equation}{}{}
\crefformat{equation}{(#2#1#3)} 
\Crefformat{equation}{(#2#1#3)}

\theoremstyle{nonumberplain}
\theoremsymbol{\ensuremath{\blacksquare}}

\newtheorem{proof}{Proof}

\def\polhk#1{\setbox0=\hbox{#1}{\ooalign{\hidewidth
    \lower1.5ex\hbox{`}\hidewidth\crcr\unhbox0}}}

\newcommand{\bes}{\begin{equation*}}
\newcommand{\ees}{\end{equation*}}
\newcommand{\be}{\begin{equation}}
\newcommand{\ee}{\end{equation}}

\allowdisplaybreaks

\usetikzlibrary{cd}

\begin{document}

\baselineskip=15pt

\title{On morphisms of topological quivers}
\author{
  Mariusz Tobolski
}

\date{}

\newcommand{\Addresses}{{
  \bigskip
  \footnotesize

  \textsc{Instytut Matematyczny, Uniwersytet Wroc\l{}awski, pl. Grunwaldzki 2/4, 50-384 Poland}\par\nopagebreak \textit{E-mail address}:
  \texttt{mariusz.tobolski@math.uni.wroc.pl}
}}

\maketitle

\begin{abstract}
We introduce regular morphisms of topological quivers and show that they give rise to a subcategory of the category of topological quivers and quiver morphisms. Our regularity conditions render the topological quiver C*-algebra construction a contravariant functor from the category of topological quivers and regular morphisms into the category of C*-algebras and $*$-homomorphisms. 
\end{abstract}


\vspace{.5cm}



\section{Introduction}

Directed graphs and their topological generalizations provide a rich framework for constructing operator algebras with deep connections to dynamics, geometry, and representation theory. Classical graph C*-algebras, first introduced in the context of discrete graphs by Kumjian, Pask, Raeburn, and Renault in~\cite{kprr97}, have been extensively studied and classified in terms of graph-theoretic data (see, e.g., \cite{raeburn}). The notion of a \emph{topological quiver}, pioneered by Muhly and Tomforde \cite{mt-05}, extends this theory by allowing vertex and edge spaces to carry locally compact Hausdorff topologies, with source and range maps continuous and structure maps given by a continuous family of measures.

While the construction and analysis of C*-algebras associated to topological quivers have been developed in recent years, the study of \emph{morphisms} between topological quivers has received less attention. Morphisms provide functoriality, enabling one to compare and relate the C*-algebras of different quivers, to construct equivalences and dualities, and to transfer structural properties across categories. This paper initiates a systematic exploration of morphisms of topological quivers, aiming to identify the right notion. Appropriate morphisms of topological graphs of Katsura~\cite{katsura1}, which are special kinds of topological quivers, were defined in~\cite{katsura2}. The special case of an isomorphism of topological quivers was considered in~\cite{l-h24}.

Our main contributions are as follows:
\begin{enumerate}[\quad(1)]
\item We introduce the notion of a \emph{regular quiver morphism} between topological quivers, requiring compatible continuous maps on vertex and edge spaces that intertwine the source, range, and measure data (Definition~\ref{admquiv}).
\item We show that regular quiver morphisms compose naturally and form a category $\mathbf{Q}_{\rm reg}$ (Proposition~\ref{qcat}). We also prove that our regular quiver morphisms generalize regular factor maps of Katsura (Section~\ref{apptop}).
\item Finally, we demonstrate that regular quiver morphism induce $*$-homomorphisms between the associated C*-algebras (Theorem~\ref{mainthm}).
\end{enumerate}

Throughout, all topological spaces are assumed to be locally compact, Hausdorff, and second countable. We denote by $C_0(X)$ the algebra of continuous complex-valued functions vanishing at infinity on a space $X$.


\section{Preliminaries}

We refer the reader to~\cite{gj-m90} for an overview of C*-algebra theory. In what follows, we make use of the following categories:
\[
{\bf C}^*:~\text{C*-algebras and $*$-homomorphisms},
\]
\[
{\bf C}^*_\mathbb{T}:~\text{$\mathbb{T}$-C*-algebras and $\mathbb{T}$-equivariant $*$-homomorphisms}.
\]
Here a $\mathbb{T}$-C*-algebra is a C*-algebra $A$ equipped with a circle (continuous) action $\alpha:\mathbb{T}\to {\rm Aut}(A)$. A $*$-homomorphism between $\mathbb{T}$-C*-algebras is $\mathbb{T}$-equivariant if it intertwines the respective actions.

\subsection{C*-correspondences and Cuntz--Pimsner algebras}

The general theory of Hilbert modules and C*-correspondences is presented in detail in~\cite{ec-l95}. The foundational idea of associating C*-algebras to C*-correspondences originates from Pimsner’s work~\cite{pimsner}. In this paper, we follow the approach developed in~\cite{katsura0} and~\cite{rs-11}.

A \emph{$C^*$-correspondence} over a $C^*$-algebra $A$ consists of a (right) Hilbert $A$-module $X$ equipped with a~nondegenerate left action of $A$, namely, a $*$-homomorphism $\varphi_X: A \to \mathcal{L}(X)$, where $\mathcal{L}(X)$ denotes the algebra of adjointable operators on $X$. We often write $(X,A)$ to indicate that $X$ is a~correspondence over $A$, and denote the left action of $a \in A$ on $\xi \in X$ by $a \cdot \xi := \varphi_X(a)(\xi)$. 

The C*-algebra $\mathcal{K}(X)$ of \emph{compact operators on $X$} is defined as the closed linear span of rank-one operators $\theta_{\xi,\eta}$ given by
\[
\theta_{\xi,\eta}(\zeta) := \xi \cdot \langle \eta, \zeta \rangle, \qquad \zeta \in X, \quad \xi, \eta \in X.
\]
This forms a closed two-sided ideal in $\mathcal{L}(X)$. An important ideal associated to a $C^*$-correspondence $X$ is
\[
J_X := \left\{ a \in A : \varphi_X(a) \in \mathcal{K}(X) \ \text{and} \ ab = 0 \ \text{for all} \ b \in \ker \varphi_X \right\},
\]
which plays a key role in the construction of the associated Cuntz--Pimsner algebra.

Finally, let $X$ and $Y$ be $C^*$-correspondences over $C^*$-algebras $A$ and $B$, respectively, and suppose that $\psi: X \to Y$ is a continuous linear map. Then one can define a $*$-homomorphism $\psi^{(1)}: \mathcal{K}(X) \to \mathcal{K}(Y)$ on rank-one operators by
\[
\psi^{(1)}(\theta_{\xi, \eta}) := \theta_{\psi(\xi), \psi(\eta)}.
\]

\begin{definition}\label{cormor}(\cite[Definition~2.3]{rs-11}, cf.~\cite[Definition~2.3]{katsura01})
Let $(X,A)$ and $(Y,B)$ be C*-cor\-re\-spon\-dences. A pair $(\psi,\pi)$ consisting of a linear map $\psi:X\to Y$ and a~$*$-ho\-mo\-morphism $\pi:A\to B$ is a~{\em correspondence morphism} from $X$ to $Y$ if
\begin{enumerate}
\item[(C1)] $\psi(a\cdot\xi)=\pi(a)\cdot\psi(\xi)$ for all $a\in A$ and $\xi\in X$,
\item[(C2)] $\pi(\langle\xi,\eta\rangle)=\langle\psi(\xi),\psi(\eta)\rangle$ for all $\xi,\eta\in X$.
\end{enumerate}
A correspondence morphism $(\psi,\pi)$ is called {\em covariant} if
\begin{enumerate}
\item[(C3)] $\pi(J_X)\subseteq J_Y$, and
\item[(C4)] $\varphi_Y(\pi(a))=\psi^{(1)}(\varphi_X(a))$ for all $a\in J_X$.
\end{enumerate}
\end{definition}
We obtain the following categories:
\[
{\bf Corr}^*:\text{C*-correspondences and correspondence morphisms},
\]
\[
{\bf Corr}^*_{\rm cov}:\text{C*-correspondences and covariant correspondence morphisms}.
\]
Every C*-algebra $A$ can be viewed as a C*-correspondence over itself with the inner product defined by the formula $\langle a,b\rangle:=a^*b$ and the left and right actions given by the left and right multiplication on $A$, respectively. Furthermore, $\mathcal{K}(A)\cong A$ via a canonical isomorphism. Hence, the category ${\bf C}^*$ can be viewed as a subcategory of both ${\bf Corr}^*$ and ${\bf Corr}^*_{\rm cov}$.

A~{\em (covariant) representation} of a C*-correspondence $X$ over $A$ on a C*-algebra $B$ is a (covariant) correspondence morphism from $X$ to a C*-algebra $B$ viewed as a C*-correspondence over itself.
\begin{definition}(\cite[Definition~3.5]{katsura0}, cf.~\cite{pimsner})
The {\em Toeplitz--Pimsner algebra} $\mathcal{T}_X$ associated to a~correspondence $X$ over $A$ is the C*-algebra $C^*(t_X,t_A)$ generated by the universal representation $(t_X,t_A)$. The {\em Cuntz--Pimsner algebra} $\mathcal{O}_X$ associated to a correspondence $X$ over $A$ is the C*-algebra $C^*(u_X,u_A)$ generated by the universal covariant representation $(u_X,u_A)$ of $X$. 
\end{definition}
By universality, for every covariant representation $(\psi,\pi)$ of $X$ on $B$, there is a~$*$-ho\-mo\-mor\-phism $\mathcal{O}_{\psi,\pi}:\mathcal{O}_X\to B$ onto $C^*(\psi,\pi)$ such that $\mathcal{O}_{\psi,\pi}\circ u_X=\psi$ and $\mathcal{O}_{\psi,\pi}\circ u_A=\pi$. The formulas
\begin{equation}\label{cuntzgauge}
\alpha_z(u_A(a))=u_A(a),\qquad \alpha_z(u_X(\xi))=zu_X(\xi),\qquad z\in\mathbb{T},\quad a\in A,\quad \xi\in X,
\end{equation}
give rise to a circle action $\alpha:\mathbb{T}\to{\rm Aut}(\mathcal{O}_X)$ on the Cuntz--Pimsner algebra of a C*-correspondence $X$ over $A$, called the~{\em gauge action}.

\begin{proposition}{\rm (\cite[Proposition 2.9 and 2.10]{rs-11})}\label{quivcorrfunct}
Let $(\psi,\pi)$ be a covariant correspondence morphism from $(X,A)$ to $(Y,B)$. Then, there exists a unique gauge-equivariant $*$-homomorphism $\Psi:\mathcal{O}_X\to\mathcal{O}_Y$ such that \[
\Psi(u_X(\xi))=u_Y(\psi(\xi)),\quad\xi\in X,\qquad \Psi(u_A(a))=u_B(\pi(a)),\qquad a\in A.
\] 
Furthermore, the assignments $(X,A)\mapsto\mathcal{O}_X$ and $(\psi,\pi)\mapsto\Psi$ give rise to a covariant functor from the category ${\bf Corr}^*_{\rm cov}$ to the category ${\bf C}^*_\mathbb{T}$.
\end{proposition}

\subsection{Topological quivers}

We refer the reader to~\cite{mt-05} for an extensive study of topological quivers and their C*-algebras. A~{\em topological quiver} $E$ is a~quintuple $(E^0,E^1,s_E,r_E,\lambda_E)$, where $E^0$ and $E^1$ are second-countable locally compact Hausdorff spaces, $s_E,r_E:E^1\to E^0$ are continuous maps such that $r_E$ is open, and $\lambda_E=\{\lambda_E^v\}_{v\in E^0}$ is a family of Radon measures on $E^1$ satisfying the following conditions:
\begin{enumerate}
\item ${\rm supp}\,\lambda_E^v=r_E^{-1}(v)$ for all $v\in E^0$,
\item the function 
\[
v\longmapsto \int_{r_E^{-1}(v)}\xi(e)d\lambda^v_E(e)
\]
belongs to $C_c(E^0)$ for all $\xi\in C_c(E^1)$.
\end{enumerate} 
We call $E^0$ the {\em vertex space}, $E^1$ the {\em edge space}, and $s_E$ and $r_E$ the {\em source} and the {\em range} map, respectively. Since $E^0$ and $E^1$ are second-countable locally compact Hausdorff spaces and $r_E$ is an open map, a system of measures as above (often called an $r_E$-system) always exists by~\cite[Corollary B.18]{d-w19} but it is not unique. In fact, one may chose the measures $\lambda^v_E$ to be probability measures. However, it is not always desirable to do this, for instance, when dealing with counting measures. Examples of topological quivers include directed graphs, topological graphs~\cite{katsura1}, locally compact groupoids~\cite{j-r80}, branched coverings~\cite{dm-01}, and topological relations~\cite{b-b04}.

One defines the following subsets of $E^0$:
\begin{align*}
E_{\rm sink}^0=&E^0\setminus \overline{s_E(E^1)},\\
E_{\rm fin}^0=&\{v\in E^0~:~\text{there is a precompact neighborhood $V$ of $v$ such that}\\
&~~s_E^{-1}(\overline{V})\text{ is compact and $r_E|_V$ is a local homeomorphism}\},\\
E_{\rm reg}^0=&E^0_{\rm fin}\setminus \overline{E^0_{\rm sink}},\qquad E_{\rm sing}^0=E^0\setminus E^0_{\rm reg}.
\end{align*}
The elements of $E_{\rm sink}^0$ are called {\em sinks}, the elements of $E_{\rm reg}^0$ are called {\em regular}, and those of $E^0_{\rm sing}$ are called {\em singular}.

Topological quivers give rise to C*-correspondences as follows. Let $A_E:=C_0(E^0)$ and define an $A$-valued inner product on $C_c(E^1)$ by the formula
\[
\langle\xi,\eta\rangle(v):=\int_{r_E^{-1}(v)}\overline{\xi(e)}\eta(e)d\lambda^v_E(e).
\]
Next, let $X_E$ denote the closure of $C_c(E^1)$ with respect to the norm $\|\xi\|:=\sup_{v\in E^0}\sqrt{\langle\xi,\xi\rangle(v)}$. Define the left and right actions of $A_E$ on $X_E$ by setting
\[
(\xi\cdot f)(e):=\xi(e)f(r_E(e)),\qquad e\in E^1,\quad \xi\in C_c(E^1),\quad f\in C_0(E^0),
\]
\[
(\varphi_E(f)\xi)(e)=(f\cdot\xi)(e):=f(s_E(e))\xi(e),\qquad e\in E^1,\quad\xi\in C_c(E^1),\quad f\in C_0(E^0),
\]
and extending it to $X_E$. The pair $(X_E,A_E)$ is the {\em quiver correspondence} associated with the topological quiver $E$. The {\em topological quiver C*-algebra} $C^*(E)$ is the Cuntz--Pimsner algebra associated to the C*-correspondence $(X_E,A_E)$.


\section{Functoriality of C*-algebras of topological quivers}

\subsection{Morphisms of topological quivers}\label{morquiv}

The notion of a graph homomorphism can be generalized to topological quivers almost verbatim.
\begin{definition}
A {\em quiver morphism} $m=(m^1,m^0)$ between topological quivers $E$ and $F$ consists of continuous maps $m^1:E^1\to F^1$ and  $m^0:E^0\to F^0$ making the diagrams
\begin{equation}\label{srdiag}
\begin{tikzcd}
E^1 \arrow[r,"s_E"] \arrow[d,"m^1"'] & E^0\arrow[d,"m^0"]\\
F^1 \arrow[r,"s_F"] & F_0 
\end{tikzcd}
\qquad\text{and}\qquad
\begin{tikzcd}
E^1 \arrow[r,"r_E"] \arrow[d,"m^1"'] & E^0\arrow[d,"m^0"]\\
F^1 \arrow[r,"r_F"] & F_0 
\end{tikzcd}
\end{equation}
commute.
\end{definition}
It is clear that we obtain a category:
\[
{\bf Q}: \text{topological quivers and quiver morphisms}.
\]

To define a special type of quiver morphisms generalizing admissible graph homomorphisms~\cite[Section~2]{th-24} and factor maps~\cite[Definition~2.6]{katsura2}, we need to recall some definitions and notation:
\begin{itemize}
\item A continuous map $f:X\to Y$ between locally compact spaces is {\em proper} if $f^{-1}(K)$ is compact for every compact $K\subseteq Y$.
\item  If $f:X\to Y$ is a measurable map between measurable spaces, then let $f_*(\mu)$ denote the pushforward of a measure $\mu$ on $X$ along the map $f$.
\item Let $f:X\to Z$ and $g:Y\to Z$ be continuous maps. Then the space
\[
X{}_f\times_gY:=\{(x,y)\in X\times Y~:~f(x)=g(y)\}\subseteq X\times Y,
\]
together with the maps $(x,y)\mapsto x$ and $(x,y)\mapsto y$ is the pullback of $f$ and $g$ in the category of topological spaces.
\end{itemize}
We are now ready to introduce regular morphisms.
\begin{definition}\label{admquiv}
Let $m:E\to F$ be a quiver morphism. We say that $m$ is {\em regular} if
\begin{enumerate}
\item[(A1)] $m^1$ and $m^0$ are proper,
\item[(A2)] $m^1|_{r_E^{-1}(v)}$ is injective and $\lambda^{m^0(v)}_F=m_*^1(\lambda_E^v)$ for all $v\in E^0$,
\item[(A3)] $(m^0)^{-1}(F^0_{\rm reg})\subseteq E^0_{\rm reg}$.
\end{enumerate}
\end{definition}
Due to the condition (A2), we have the following equality
\begin{equation}\label{a2con}
\int_{r_F^{-1}(m^0(v))}\xi(x)d\lambda^{m^0(v)}_F(x)=\int_{r_E^{-1}(v)}\xi(m^1(e))d\lambda_E^v(e)
\end{equation}
for every $v\in E^0$ and $\xi\in C_c(F^1)$, which we will use throughout the paper. 
\begin{remark}
Our notion of a regular morphism generalizes regular factor maps of Katsura~\cite[Definition~2.6]{katsura2} and admissible graph homomorphism in the sense of~\cite[Section~2]{th-24}. See Section~\ref{apptop} for details. We also note that the second part of the condition (A2) is analogous to~\cite[Definition~3.1]{ag-19} introduced in the context of locally compact groupoids.
\end{remark}

\begin{proposition}\label{qcat}
Topological quivers together with regular quiver morphisms form a subcategory of the category ${\bf Q}$.
\end{proposition}
\begin{proof}
Clearly, ${\rm id}_E=({\rm id}_{E^0},{\rm id}_{E^1})$ is a regular quiver morphism.
Next, let $m:E\to F$ and $n:F\to G$ be regular quiver morphisms. We have to prove that $n\circ m$ is also a regular quiver morphism.
Since composition of proper maps is again proper, we only need to check the conditions (A2) and (A3) of Definition~\ref{admquiv}. 

To prove the first part of the condition (A2), take $v\in E^1$ and suppose that $n^1(m^1(e))=n^1(m^1(e'))$ for some $e,e'\in r_E^{-1}(v)$. Since $m^1(e),m^1(e')\in r_F^{-1}(m^0(v))$, the first part of (A2) for $n^1$ implies that $m^1(e)=m^1(e')$. Hence, the result follows from the same condition for $m^1$. Note that the second part of the condition (A2) follows from a basic property of the pushforward of a~measure along a~composition of maps, namely
\[
(n^1\circ m^1)_*(\lambda_E^v)=n^1_*(m^1_*(\lambda_E^v))
=n^1_*(\lambda_F^{m^0(v)})=\lambda_G^{(n^0\circ\, m^0)(v)},\qquad v\in E^0.
\]
Proving (A3) is also straightforward:
\[
(m^0)^{-1}((n^0)^{-1}(G^0_{\rm reg}))\subseteq (m^0)^{-1}(F^0_{\rm reg})\subseteq E^0_{\rm reg}.
\]
\end{proof}
We introduce the following notation:
\[
{\bf Q}_{\rm reg}:~\text{topological quivers and regular  quiver morphisms}.
\]

\subsection{Relation to factor maps and admissible graph homomorphisms}\label{apptop}

Let us recall the notion of a topological graph introduced by Katsura in~\cite{katsura1} (note that we use different convention, i.e. the source and range maps are interchanged). A {\em topological graph} $E$ is a~quadruple $(E^1,E^1,s_E,r_E)$, where $E^0$ and $E^1$ are locally compact topological spaces, $r_E:E^1\to E^0$ is a local homeomorphism and $s_E:E^1\to E^0$ is a continuous map. 

Next, we discuss the relation between topological graphs and topological quivers. Let $E^1$ and $E^0$ be second countable. For each $v\in E^0$ and every Borel set $B$ in $E^1$, we consider the counting measures
\[
\lambda^v_E(B):=\begin{cases} |B\cap r_E^{-1}(v)| & \text{if }B\cap r_E^{-1}(v)\text{ is finite}\\ \infty & \text{otherwise.}\end{cases}
\]
Since $r_E^{-1}(v)$ is countable, the above formula gives rise to a Radon measure on $E^1$ with support $r_E^{-1}(v)$. Then $(E^1,E^0,s_E,r_E,\lambda_E)$, where $\lambda_E:=\{\lambda^v_E\}_{v\in E^0}$, defines a topological quiver. Let us also mention that Katsura's definition of the set of regular vertices $E^0_{\rm reg}$ coincides with the same notion for topological quivers in the case of topological graphs.

Let $\widetilde{X}:=X\cup\{\infty\}$ denotes the one-point compactification of a locally compact space $X$.
\begin{definition}{\cite[Definition~2.1]{katsura2}}\label{factor}
A {\em factor map} $\widetilde{m}$ from a topological graph $E$ to a topological graph $F$ is a pair $(\widetilde{m}^1,\widetilde{m}^0)$ consisting of continuous maps $\widetilde{m}^0:\widetilde{E}^0\to\widetilde{F}^0$ and $\widetilde{m}^1:\widetilde{E}^1\to\widetilde{F}^1$ sending $\infty$ to $\infty$ and such that
\begin{enumerate}
\item[(F1)] For every $e\in E^1$ with $\widetilde{m}^1(e)\in F^1$, we have that $s_F(\widetilde{m}^1(e))=\widetilde{m}^0(s_E(e))$ and $r_F(\widetilde{m}^1(e))=\widetilde{m}^0(r_E(e))$. 
\item[(F2)] If $x\in F^1$ and $v\in E^0$ satisfies $r_F(x)=\widetilde{m}^0(v)$, then there exists a unique element $e\in E^1$ such that $\widetilde{m}^1(e)=x$ and $r_E(e)=v$.
\end{enumerate}
\end{definition}
\begin{definition}{\cite[Definition~2.6]{katsura2}}\label{regfactor}
A factor map $\widetilde{m}:E\to F$ is called {\em regular} if $s_E^{-1}(v)$ is non-empty and contained in $(\widetilde{m}^1)^{-1}(F^1)$ for every $v\in E^0$ with $\widetilde{m}^0(v)\in F^0_{\rm reg}$.
\end{definition}

Let us prove the equivalence of the notion of a regular factor map and a regular quiver morphism in the setting of topological graphs.
\begin{proposition}
Let $E$ and $F$ be topological graphs. A pair $(\widetilde{m}^1,\widetilde{m}^0)$ is a regular factor map if and only if $(m^1,m^0)$ is a regular quiver morphism.
\end{proposition}
\begin{proof}
First, note that the condition (F1) of Definition~\ref{factor} for $(\widetilde{m}^1,\widetilde{m}^0)$ is equivalent to $(m^1,m^0)$ being a quiver morphism. Next, for locally compact Hausdorff spaces $X$ and $Y$, it is well-known that if $\widetilde{f}:\widetilde{X}\to\widetilde{Y}$ is a continuous map sending $\infty$ to $\infty$, then $f:=\widetilde{f}|_X$ is a proper continuous map from $X$ to $Y$. Conversely, if $f:X\to Y$ is a proper continuous map, then extending $f$ to $\widetilde{f}$ by putting $\widetilde{f}(\infty)=\infty$ gives rise to a continuous map from $\widetilde{X}$ to $\widetilde{Y}$. Hence, $(\widetilde{m}^1,\widetilde{m}^0)$ is a pair of continuous maps sending $\infty$ to $\infty$ if and only if $(m^1,m^0)$ is a pair of proper continuous maps, i.e. satisfying the condition (A1) of Definition~\ref{admquiv}. 

Next, assume the condition (F2) of Definition~\ref{factor}. Then, clearly $m^1$ restricted to ${r_E^{-1}(v)}$ is injective for every $v\in E^0$. We need to prove that $\lambda_F^{m^0(v)}=m^1_*(\lambda_E^v)$ for every $v\in E^0$. Recall that $m^1_*(\lambda^v_E)(B)=\lambda_E^v((m^1)^{-1}(B))$ for every Borel set $B\subseteq E^1$. Due to the condition (F2), we have that $|B\cap r_F^{-1}(m^0(v))|=|(m^1)^{-1}(B)\cap r^{-1}_E(v)|$ so the claim follows. Now, assume the condition (A2) of Definition~\ref{admquiv}. Take any $x\in F^1$ and $v\in E^0$ such that $r_F(x)=m^0(v)$. Then 
\[
1=\lambda_F^{m^0(v)}(x)=m^1_*(\lambda_E^v)(x)=\lambda_E^v((m^1)^{-1}(x))=|(m^1)^{-1}(x)\cap r_E^{-1}(v)|.
\]
Hence, there exists a unique $e\in E^1$ such that $m^1(e)=x$ and $r_E(e)=v$.

Finally, let the factor map $(\widetilde{m}^1,\widetilde{m}^0)$ be regular. Then $(m^1,m^0)$ satisfies the condition (A3) due to~\cite[Lemma~2.7]{katsura2}. Conversely, the condition (A3) of Definition~\ref{admquiv} clearly implies regularity of $(\widetilde{m}^1,\widetilde{m}^0)$.
\end{proof}

\subsection{Contravariant functor from quivers to correspondences}

In this section, we show that regular admissible quiver morphisms between topological quivers lead to $*$-homomorphisms between their quiver C*-algebras. First, we need some preliminary results.
\begin{lemma}\label{quivcorr}
Let $m:E\to F$ be an regular quiver morphism. Then the assignments
\[
\mu^0:A_F\longrightarrow A_E,\qquad f\longmapsto f\circ m^0,
\]
 \[
 \mu^1:C_c(F^1)\longrightarrow C_c(E^1),\qquad \xi\longmapsto \xi\circ m^1,
 \]
give rise to a correspondence morphism $(\mu^1,\mu^0):(X_F,A_F)\to (X_E,A_E)$.
\end{lemma}
\begin{proof}
Since $m^1$ and $m^0$ are proper, $\mu^1$ and $\mu^0$ are well defined. Routine calculations show that $\mu^1$ is linear and that $\mu^0$ is a $*$-homomorphism.
The equation~\eqref{a2con} implies that
\begin{align}
\langle\mu^1(\xi),\mu^1(\eta)\rangle(v)&=\int_{r_E^{-1}(v)}\overline{\xi(m^1(e))}\eta(m^1(e))d\lambda_E^v(e)=\int_{r_F^{-1}(m^0(v))}\overline{\xi(x)}\eta(x)d\lambda_F^{m^0(v)}(x) \nonumber\\
&=\langle\xi,\eta\rangle(m^0(v))=\mu^0(\langle\xi,\eta\rangle)(v).\label{c2proof}
\end{align}
We infer that
\[
\|\mu^1(\xi)\|^2=\sup_{v\in E^0}\langle\mu^1(\xi),\mu^1(\xi)\rangle(v)=\sup_{v\in E^0}\langle\xi,\xi\rangle(m^0(v))=\sup_{w\in m^0(E^0)}\langle\xi,\xi\rangle(w)\leq \|\xi\|^2.
\]
Hence, $\mu^1$ is continuous and extends uniquely to a continuous linear map $X_F\to X_E$ which we denote by the same symbol. Due to continuity of $\mu^1$ and the inner product, \eqref{c2proof} also shows that the pair $(\mu^1,\mu^0)$ satisfies the condition (C2) of Definition~\ref{cormor}. Next, we check the condition (C1):
\begin{align*}
\mu^1(f\cdot\xi)(e)&=(f\cdot\xi)(m^1(e))=f(s_F(m^1(e)))\xi(m^1(e))\\&=f(m^0(s_E(e))\xi(m^1(e))=(\mu^0(f)\cdot\mu^1(\xi))(e),
\end{align*}
where $e\in E^1$, $f\in C_0(F^0)$, and $\xi\in C_c(F^1)$. By the continuity of $\mu^1$ and the left action, we conclude that (C1) holds for all elements of $X_F$. 
\end{proof}
\begin{remark}
Note that we did not need to use the condition (A3) of Definition~\ref{admquiv} to prove the above lemma.
\end{remark}

The next lemma is a technical result needed for proving that the correspondence morphism defined in Lemma~\ref{quivcorr} is covariant.
\begin{lemma}\label{c4lemma}
Let $m:E\to F$ be an regular quiver morphism and let $(\mu^1,\mu^0)$ be its associated correspondence morphism. If $g\in C_0(F^1)$ is such that for every $x\in {\rm osupp}\,g$ there is an open neighborhood $U_x$ of $x$ such that $r_F|_{U_x}:U_x\to r_F(U_x)$ is a homeomorphism, then
\[
(\mu^1)^{(1)}(\sigma_F(g))=\sigma_E(\mu^1(g)).
\]
Here $\sigma_E$ and $\sigma_F$ are the $*$-homomorphisms defined in~\cite[Lemma~3.6]{mt-05}.
\end{lemma}
\begin{proof}
First, recall that for every $g\in C_0(F^1)$ and $\varepsilon>0$ there is a function $\zeta\in C_c(F^1)$ which is $\varepsilon$-close to $g$ in sup-norm and such that ${\rm osupp}\,\zeta\subseteq{\rm osupp}\,g$ (see, e.g.~\cite[Lemma~3.10]{mt-05}). Since $(\mu^1)^{(1)}\circ\sigma_F$ and $\sigma_E\circ\mu^1$ are continuous, it suffices to prove the claim for $\zeta$ as above. 
Using our assumptions on $g$ and arguing as in the proof of~\cite[Theorem~3.11]{mt-05} (see also~\cite[Lemma~1.15 and 1.16]{katsura1}), we conclude that $\sigma_F(\zeta)=\sum_{i=1}^m\theta_{\xi_i,\eta_i}$ for $\xi_i,\eta_i\in C_c(F^1)$, $i=1,\ldots,m$, defined as follows:
\begin{itemize}
\item Since $\zeta$ is compactly supported, there exist finitely many $x_1,\ldots,x_m\in F^1$ and their open neighborhoods $U_1,\ldots, U_m$ such that $r_F$ is injective when restricted to them and such that ${\rm supp}\,\zeta\subseteq\bigcup_{i=1}^mU_i$. 
\item Next, let $\chi_i$ be the compactly supported partition of unity associated to the open cover $\{U_i\}_{i=1}^m$.
\item Finally, put $\xi_i(x):=\zeta(x)\sqrt{\chi_i(x)}$ and  $\eta_i(x):=\sqrt{\chi_i(x)}\lambda^{r_F(x)}_F(\{x\})^{-1}$ for all $x\in F^1$ and $i=1,\ldots, m$. One can prove that $\xi_i,\eta_i\in C_c(F^1)$ and that ${\rm supp}\,\xi_i={\rm supp}\,\eta_i\subseteq U_i$, $i=1,\ldots, m$. Furthermore, we have that
\begin{equation}\label{partint}
\xi_i(x)\int_{r_F^{-1}(r_F(x))}\overline{\eta_i(x')}h(x')d\lambda_F^{r_F(x)}(x')=\xi_i(x)\sqrt{\chi_i(x)}h(x)
\end{equation}
for every $x\in F^1$ and $h\in C_c(F^1)$.
\end{itemize} 
Note that $r_E^{-1}(r_E(e))\subseteq E^1$ is closed for every $e\in E^1$ and that $m^1$ restricted to $r_E^{-1}(r_E(e))$ is injective by the condition (A2). Take any $f\in C_c(E^1)$ and $e\in E^1$. Then, by Tietze extension theorem, there is $h_e\in C_c(F^1)$ such that $f=h_e\circ m^1$ on $r_E^{-1}(r_E(e))$. Therefore, for any $f\in C_c(E^1)$, we have that
\begin{align*}
\left((\mu^1)^{(1)}(\sigma_F(\zeta))f\right)(e)&=\left((\mu^1)^{(1)}\left(\sum_{i=1}^m\theta_{\xi_i,\eta_i}\right)f\right)(e)=\left(\sum_{i=1}^m\theta_{\mu^1(\xi_i),\mu^1(\eta_i)}f\right)(e)\\
&=\sum_{i=1}^m(\mu^1(\xi_i))(e)\langle\mu^1(\eta_i),f\rangle(r_E(e))\\
&=\sum_{i=1}^m\xi_i(m^1(e))\int_{r_E^{-1}(r_E(e))}\overline{\eta_i(m^1(e'))}h_e(m^1(e'))d\lambda^{r_E(e)}_E(e')\\
&=\sum_{i=1}^m\xi_i(m^1(e))\int_{r_F^{-1}(r_F(m^1(e)))}\overline{\eta_i(x)}h_e(x)d\lambda^{r_F(m^1(e))}_F(x)\\
&=\sum_{i=1}^m\xi_i(m^1(e))\sqrt{\chi_i(m^1(e))}f(e)=\zeta(m^1(e))f(e)=(\sigma_E(\mu^1(\zeta))f)(e).
\end{align*}
Here we used~\eqref{a2con} and~\eqref{partint}.
\end{proof}
\begin{lemma}\label{quivcov}
Let $m:E\to F$ be a regular quiver morphism. Then the correspondence morphism $(\mu^1,\mu^0)$ defined in Lemma~\ref{quivcorr} is covariant.
\end{lemma}
\begin{proof}
To prove the condition (C3) of Definition~\ref{cormor}, recall that $J_{X_E}=C_0(E^0_{\rm reg})$ for any quiver correspondence $(X_E,A_E)$. Hence, (C3) follows directly from the condition (A3) of Definition~\ref{admquiv}. Indeed, if $f\in C_0(F^0_{\rm reg})$ then $f(w)=0$ for all $w\notin F^0_{\rm reg}$. Now, the condition (A3) implies that $m^0(v)\notin F^0_{\rm reg}$ for any $v\notin E^0_{\rm reg}$. Hence, $\mu^0(f)(v)=f(m^0(v))=0$ for any $v\notin E^0_{\rm reg}$, which implies that $\mu^0(f)\in C_0(E^0_{\rm reg})$ for all $f\in C_0(F^0_{\rm reg})$. Finally, we prove the condition (C4). Due to~\cite[Theorem 3.11]{mt-05}, Lemma~\ref{c4lemma}, and the fact that $\varphi_E(f)=\sigma_E(f\circ s_E)$, it suffices to prove that
\[
\mu^1(f\circ s_F)=\mu^0(f)\circ s_E
\]
for all $f\in C_0(E^0_{\rm reg})$. If $s_E(e)\notin E^0_{\rm reg}$, then $m^0(s_E(e))\notin F^0_{\rm reg}$, so that $\mu^0(f)(s_E(e))=0=f(m^0(s_E(e)))$. Therefore, we have that
\[
\mu^1(f\circ s_F)(e)=f(s_F(m^1(e)))=f(m^0(s_E(e)))=\mu^0(f)(s_E(e)),\qquad e\in E^1.
\]

\end{proof}

\begin{theorem}\label{mainthm}
Let $m:E\to F$ be an regular quiver morphism. Then  there exists a gauge-equivariant $*$-homomorphism $\mu:\mathcal{O}_F\to\mathcal{O}_E$ such that the assignments
\[
E\longmapsto \mathcal{O}_E,\qquad m\longmapsto \mu,
\]
give rise to a contravariant functor from the category ${\bf Q}_{\rm reg}$ to the category ${\bf C}^*_{\mathbb{T}}$ of $\mathbb{T}$-C*-algebras and $\mathbb{T}$-equivariant $*$-homomorphisms. If $E^0$ and $F^0$ are compact, then $\mu$ is unital.
\end{theorem}
\begin{proof}
First, note that the assignments $E\mapsto (X_E,A_E)$ and $(m^1,m^0)\mapsto (\mu^1,\mu^0)$ define a contravariant functor from ${\bf Q}_{\rm reg}$ to ${\bf Corr}^*_{\rm cov}$. Indeed, combining Lemma~\ref{quivcorr} and Lemma~\ref{quivcov}, we have that $(\mu^1,\mu^0)$ is a covariant correspondence morphism. Hence, the aforementioned assignments are well defined. Next, since the assignment $m^0\mapsto \mu^0$ is simply the Gelfand--Naimark contravariant functor and showing that the assingment $m^1\mapsto \mu^1$ gives a contravariant functor is straightforward. Finally, we compose the above functor with the covariant functor from the category ${\bf Corr}_{\rm cov}^*$ to the category ${\bf C}_\mathbb{T}^*$ described in~Proposition~\ref{quivcorrfunct} to obtain the desired contravariant functor from $\bf{Q}_{\rm reg}$ to $\bf{C}^*_\mathbb{T}$. 
\end{proof}

\section*{Acknowledgements} 
\noindent
This work is part of the project “Applications of graph algebras and higher-rank graph algebras
in noncommutative geometry” partially supported by NCN grant UMO-2021/41/B/ST1/03387. The author would like to thank Matthew Gillaspie and Ben Jones for discussions on morphism of topological quivers.


\bibliographystyle{plain}
\bibliography{join}

\Addresses

\end{document}